\documentclass[12]{amsart}
\usepackage{amsmath,amssymb}
\usepackage{amsthm,color,enumerate,comment,centernot,enumitem,url,cite}
\usepackage{graphicx,relsize,bm}
\usepackage{mathtools}
\usepackage{array}

\usepackage{tkz-graph}
\usetikzlibrary{arrows}
\usetikzlibrary{arrows.meta,
                decorations.markings}

\usepackage{enumitem}
\setenumerate[0]{label=(\alph*)}

\makeatletter
\newcommand{\tpmod}[1]{{\@displayfalse\pmod{#1}}}

\makeatother

\newtheorem{thm}{Theorem}[section]
\newtheorem{lemma}[thm]{Lemma}

\theoremstyle{definition}

\theoremstyle{remark}

\theoremstyle{definition}
    \newtheorem{defn}[thm]{Definition}

\newtheorem{rem}[thm]{Remark}

\newcommand{\abs}[1]{\left|{#1}\right|}

\def\R {{\mathbb R}}
\def\Z {{\mathbb Z}}

\def\Q {{\mathbb Q}}

\def\C {{\mathcal C}}

\def\Z {{\mathbb Z}}
\def\Q {{\mathbb Q}}
\def\C {{\mathbb C}}

\def\CC {{\mathcal C}}

\makeatletter
\@namedef{subjclassname@2020}{%
  \textup{2020} Mathematics Subject Classification}
\makeatother

\def\red#1 {\textcolor{red}{#1 }}
\def\blue#1 {\textcolor{blue}{#1 }}

\numberwithin{equation}{section}

\def\Z {{\mathbb Z}}
\begin{document}

\title[Monogenic Strictly-Perron Polynomials]{Monogenic Strictly-Perron Polynomials}

%\author{Joshua Harrington}
%\address{Department of Mathematics, Cedar Crest College, Allentown, Pennsylvania, USA}
%\email[Joshua Harrington]{Joshua.Harrington@cedarcrest.edu}

\author{Lenny Jones}
\address{Professor Emeritus, Department of Mathematics, Shippensburg University, Shippensburg, Pennsylvania 17257, USA}
\email[Lenny~Jones]{doctorlennyjones@gmail.com}

%\author{Daniel White}
%\address{Department of Mathematics, Bryn Mawr College, Bryn Mawr, Pennsylvania 19010-2899, USA}
%\email[Daniel~White]{dfwhite@brynmawr.edu}
\date{\today}

\begin{abstract} A monic polynomial $f(x)\in {\mathbb Z}[x]$ of degree $n$ is called \emph{monogenic} if $f(x)$ is irreducible over ${\mathbb Q}$ and $\{1,\theta,\theta^2,\ldots ,\theta^{n-1}\}$ is a basis for the ring of integers of ${\mathbb Q}(\theta)$, where $f(\theta)=0$. A {\em strictly-Perron} polynomial is the minimal polynomial of a Perron number $\lambda$ such that $\lambda$ is neither a Pisot number, an anti-Pisot number, nor a Salem number. 
For any natural number $n\ge 2$, we prove that there exist infinitely many monogenic strictly-Perron polynomials of degree $n$.  
   \end{abstract}

 %, power-compositional}

\maketitle
\section{Introduction}\label{Section:Intro}

Suppose that $f(x)\in {\mathbb Z}[x]$ of degree $n\ge 2$ is monic and irreducible over $\Q$. Then, we define $f(x)$ to be \emph{monogenic} if $\Theta=\{1,\theta,\theta^2,\ldots ,\theta^{n-1}\}$ is a basis for the ring of integers ${\mathbb Z}_K$ of $K={\mathbb Q}(\theta)$, where $f(\theta)=0$. Such a basis $\Theta$ is commonly referred to in the literature as a \emph{power basis}.
It is well known \cite{Cohen} that 
\begin{equation} \label{Eq:Dis-Dis}
\Delta(f)=\left[\Z_K:\Z[\theta]\right]^2\Delta(K),
\end{equation}
 where $\Delta(f)$ and $\Delta(K)$ denote, respectively, the discriminants over $\Q$ of $f(x)$ and $K$. Thus, we see that $f(x)$ is monogenic if and only if $\left[\Z_K:\Z[\theta]\right]=1$ or, equivalently, $\Delta(f)=\Delta(K)$. Hence, if $\Delta(f)$ is squarefree, then $f(x)$ is monogenic. However, the converse is false, and it is also possible that a power basis for $\Z_K$ exists even if $f(x)$ is not monogenic. For a nontrivial example, see \cite{JonesWhite}.

A {\em Perron number}, as defined by Lind in \cite{Lind1}, is a real algebraic integer $\lambda>1$ such that $\abs{\alpha}<\lambda$ for all Galois conjugates $\alpha$ of $\lambda$, with $\alpha\ne \lambda$. We caution the reader that, for technical reasons, Lind modified the original definition of a Perron number in \cite{Lind2,Lind3}, by allowing 1 to be a Perron number. In this paper, we use the original definition of a Perron number given in \cite{Lind1}, and we define a {\em Perron polynomial} to be the minimal polynomial of a Perron number. 
 The concept of a Perron number is closely related to a theorem due to Perron \cite{Perron} in which he showed that a square matrix, whose entries are all positive real numbers, has a unique dominant positive real eigenvalue $\lambda$. Thus, if $\lambda>1$, then $\lambda$ is a Perron number. 
 
 Frobenius \cite{F} subsequently extended Perron's theorem to matrices whose entries are nonnegative real numbers, provided the matrix cannot be transformed into block upper-triangular form through a permutation of rows and columns. Such a matrix is called {\em irreducible}, and {\em reducible} otherwise \cite{HJ}.  
This generalization of Perron's theorem, due to Frobenius, is known as the Perron-Frobenius Theorem \cite{HJ,Meyer}.  
 The Perron-Frobenius Theorem  has many applications \cite{Ding,Elden,K1,K2,Li,Lind1,Lind2,Lind3,Mac,Meyer,Thiemann} such as Markov chains, dynamical systems, chaos theory and internet search engines, to name a few.

The set of all Perron numbers contains all Pisot numbers \cite{Bertin} and all Salem numbers \cite{Bertin,Smyth}, which are of considerable interest themselves. A \emph{Pisot} (or \emph{Pisot-Vijayaraghavan}) number \cite{Bertin} is a real algebraic integer $\lambda>1$ such that all other zeros of the minimal polynomial $f(x)\in \Z[x]$ of $\lambda$ lie inside the unit circle. We call $f(x)$ a \emph{Pisot polynomial}, which of course is also a Perron polynomial. The monogenicity of Pisot polynomials was investigated by the author in \cite{JonesPisot}. A {\em Salem} number is a real algebraic integer $\lambda>1$ of even degree at least 4 \cite{Boyd}, such that $\lambda^{-1}$ is a Galois conjugate of $\lambda$, and all Galois conjugates of $\lambda$ other than $\lambda$ and $\lambda^{-1}$ have modulus exactly 1. 
The minimal polynomial of a Salem number is called a  {\em Salem polynomial}, 
and it is also a Perron polynomial.    

Yet another type of number that has been investigated is an {\em anti-Pisot} number \cite{SS},  which is a real algebraic integer $\lambda > 1$ such that exactly one Galois conjugate of $\lambda$ has modulus less than 1 and at least one Galois conjugate other than $\lambda$ has modulus greater than 1. An {\em anti-Pisot polynomial} is the minimal polynomial of an anti-Pisot number. Monogenic anti-Pisot polynomials were presented in \cite{JonesPisot}. 

In this article, we extend the work in \cite{JonesPisot} by proving that, for any natural number $n\ge 2$, there exist infinitely many monogenic Perron polynomials of degree $n$ that are neither Pisot, nor anti-Pisot nor Salem. We call such Perron polynomials {\em strictly-Perron} polynomials, and their associated Perron numbers {\em strictly-Perron} numbers. More precisely, our main theorem is the following:
\begin{thm}\label{Thm:Main}
Let $a,n\in \Z$, with $a\ge 1$ and $n\ge 2$. If $\gcd(a,n)=1$, then there exist infinitely many primes $p$ such that 
      \[f_{n,a,p}(x)=x^n-ax^{n-1}-p\]
   is a monogenic strictly-Perron polynomial.   
\end{thm} 
\noindent  

\begin{rem}
Since any positive integer power of a Perron number is also a Perron number \cite[Lemma 2.3]{HS},  
 the monogenicity of a Perron polynomial $f(x)$ of degree $n$, with corresponding Perron number $\lambda$, implies that the ring of integers of $\Q(\lambda)$  has the power basis $\{1,\lambda,\lambda^2,\ldots,\lambda^{n-1}\}$, where each basis element other than 1 is a Perron number.
\end{rem}

\section{Preliminaries}\label{Section:Prelim}
For the convenience of the reader, we summarize in Table \ref{T:1} the properties of Perron, Pisot, Salem and anti-Pisot numbers $\lambda$ and their Galois conjugates. 
    \begin{table}[h]
 \begin{center}
\begin{tabular}{cc}
 $\lambda$ &  Conditions on the zeros of the minimal polynomial for $\lambda$  \\ [2pt] \hline
 Perron & $\lambda>1$, $\abs{\alpha}<\lambda$ when $\alpha\ne \lambda$ \\ [2pt]
 Pisot &  $\lambda>1$, $\abs{\alpha}<1$ when $\alpha\ne \lambda$ \\ [2pt]
 Salem  & $\lambda>1$, $2\mid \deg(\lambda)\ge 4$, $\lambda^{-1}$ is a conjugate, $\abs{\alpha}=1$ when $\alpha\not \in \{\lambda,\lambda^{-1}\}$ \\ [2pt]
 anti-Pisot & $\lambda>1$, $\abs{\alpha}<1$ for exactly one conjugate $\alpha$, \\ [1pt]
  & $\abs{\beta}>1$ for at least one conjugate $\beta$ with $\beta\ne \lambda$.\\ [2pt]
 \end{tabular}
\end{center}
\caption{Properties of Perron, Pisot, Salem and anti-Pisot numbers}
 \label{T:1}
\end{table}

Throughout this paper, we let 
\begin{equation}\label{Eq:f_p(x)}
f_{n,a,p}(x):=x^n-ax^{n-1}-p,
\end{equation} where $n,a,p\in \Z$ with $n\ge 2$, $a\ge 1$ and $p$ prime.
 Using the formula for the discriminant of an arbitrary trinomial \cite{Swan}, we get 
\begin{equation}\label{Eq:Delta(f_p)}
\Delta(f_{n,a,p})=(-1)^{(n-1)(n+2)/2}p^{n-2}\left(n^{n}p+a^n(n-1)^{n-1}\right).
\end{equation}

\begin{lemma}\label{Lem:irreducible}
  Let $a_0,n\in \Z$ with $a_0\ne 0$ and $n\ge 2$. Let
  \[f(x)=x^n+a_{n-1}x^{n-1}+a_{n-2}x^{n-2}+\cdots +a_1x+a_0\in \Z[x].\] 
  \begin{enumerate}
    \item \label{Perron} If either 
    \begin{align*}
    \begin{split}
      \abs{a_{n-1}}&>1+\abs{a_{n-2}}+\abs{a_{n-3}}+\cdots +\abs{a_0},\\
      &\qquad \mbox{or} \\
      \abs{a_{n-1}}&=1+\abs{a_{n-2}}+\abs{a_{n-3}}+\cdots +\abs{a_0}, \ \mbox{with $f(\pm 1)\ne 0$},
      \end{split}
    \end{align*}
    then $f(x)$ is irreducible over $\Q$.
    \item \label{JonesActa} If $\abs{a_0}$ is prime with
  \begin{equation*}\label{Eq:condition}
  \abs{a_0}>1+\abs{a_{n-1}}+\abs{a_{n-2}}+\cdots+\abs{a_1},
  \end{equation*}
  then $f(x)$ is irreducible over $\Q$.
  \end{enumerate}
  \end{lemma}
\noindent In Lemma \ref{Lem:irreducible}, \ref{Perron} is due to Perron \cite{Perron}, while a proof of \ref{JonesActa} can be found in \cite{JonesActa}.

\begin{lemma}{\rm \cite{JonesActa}}\label{Lem:Linear}
  Suppose that $G(t)=at+b\in \Z[t]$, where $\gcd(a,b)=1$.  
  Then there exist infinitely many primes $p$ such that $G(p)$ is squarefree.
\end{lemma}

The next theorem is a theorem due to Jakhar, Khanduja and Sangwan \cite{JKS2} that gives necessary and sufficient conditions for an arbitrary irreducible trinomial $f(x)=x^{n}+A^mx+B\in \Z[x]$ to be monogenic, where $m<n$. 
\begin{thm}\label{Thm:JKS2} %{\rm \cite{JKS2}}
Let $n\ge 2$ be an integer.
Let $K=\Q(\lambda)$ be an algebraic number field with $\lambda\in \Z_K$, the ring of integers of $K$, having minimal polynomial $f(x)=x^{n}+Ax^m+B$ over $\Q$, where $\gcd(m,n)=d_0$, $m=m_1d_0$ and $n=n_1d_0$. A prime factor $q$ of $\Delta(f)$ does not divide $\left[\Z_K:\Z[\lambda]\right]$ if and only if $q$ satisfies one of the following conditions:
 \begin{enumerate}[label=(\roman*), font=\normalfont]
  \item \label{JKS:C1} when $q\mid A$ and $q\mid B$, then $q^2\nmid B$;
  \item \label{JKS:C2} when $q\mid A$ and $q\nmid B$, then
  \[\mbox{either } \quad q\mid a_2 \mbox{ and } q\nmid b_1 \quad \mbox{ or } \quad q\nmid a_2\left((-B)^{m_1}a_2^{n_1}+\left(-b_1\right)^{n_1}\right),\]
  where $a_2=A/q$ and $b_1=\frac{B+(-B)^{q^j}}{q}$, such that $q^j\mid\mid n$ with $j\ge 1$;
  \item \label{JKS:C3} when $q\nmid A$ and $q\mid B$, then
  \[\qquad \quad \mbox{either}\quad q\mid a_1 \mbox{ and } q\nmid b_2  \quad\mbox{or}\quad  q\nmid a_1b_2^{m-1}\left((-A)^{m_1}a_1^{n_1-m_1}-\left(-b_2\right)^{n_1-m_1}\right),\]
  where $a_1=\frac{A+(-A)^{q^l}}{q}$, such that $q^l\mid\mid (n-m)$ with $l\ge 0$, and $b_2=B/q$;
  \item \label{JKS:C4} when $q\nmid AB$ and $q\mid m$ with $n=s^{\prime}q^k$, $m=sq^k$, $q\nmid \gcd\left(s^{\prime},s\right)$, then the polynomials
   \begin{equation*}
     x^{s^{\prime}}+Ax^s+B \quad \mbox{and}\quad \dfrac{Ax^{sq^k}+B+\left(-Ax^s-B\right)^{q^k}}{q}
   \end{equation*}
   are coprime modulo $q$;
         \item \label{JKS:C5} when $q\nmid ABm$, then
     \[q^2\nmid \left(B^{n_1-m_1}n_1^{n_1}-(-1)^{m_1}A^{n_1}m_1^{m_1}(m_1-n_1)^{n_1-m_1}\right).\]
   \end{enumerate}
\end{thm}

\begin{defn}\cite{HJ}
  The {\em directed graph} of an $n\times n$ matrix $M$, which we denote $\Gamma(M)$, is the directed graph on $n$ vertices $v_1, v_2, \ldots ,v_n$ such that there is a directed path in $\Gamma(M)$ from $v_i$ to $v_j$ if and only if $M(i,j)\ne 0$, where $M(i,j)$ is the entry in the $i$th-row, $j$th-column of $M$. We say that $\Gamma(M)$ is {\em strongly connected} if between every pair of vertices $v_i$ and $v_j$ in $\Gamma(M)$, there is a directed path that begins at $v_i$ and ends at $v_j$. 
\end{defn}

The next two results give information concerning the irreducibility of a matrix. 
\begin{thm}\label{Thm:Irreducible} {\rm \cite{HJ}}
  Let $M$ be an $n\times n$ matrix. Then $M$ is irreducible if and only if $\Gamma(M)$ is strongly connected.  
\end{thm}
Since a reducible matrix is similar to a matrix in block upper-triangular form, we have the next lemma.
\begin{lemma}
  Let $A$ and $B$ be two similar $n\times n$ matrices. Then $A$ is irreducible if and only if $B$ is irreducible. 
\end{lemma}

The following lemma gives sufficient conditions for the characteristic polynomial of a matrix to be a Perron polynomial.
\begin{lemma}\label{Lem:PFP} %[Perron-Frobenius]{\rm \cite{HJ}}\label{P-F}
 Let $M$ be a square irreducible matrix whose entries are nonnegative integers, so that $M$ has a unique dominant positive real eigenvalue $\lambda$ by the Perron-Frobenius theorem. 
 If the characteristic polynomial of $M$, which we denote $p_M(x)$, %of $M$ 
 is irreducible over $\Q$, then $p_M(x)$ 
 is a Perron polynomial, and $\lambda$ is a Perron number.   
\end{lemma} 
\begin{proof}
    If $\lambda\le 1$, then the product of all eigenvalues of $M$ has modulus less than 1, contradicting the fact that $p_M(0)\in \Z$. Hence, $\lambda>1$. Since $p_M(x)$ is irreducible, the minimal polynomial of $\lambda$ is  $p_M(x)$, and the proof is complete. 
\end{proof}

\section{The Proof of Theorem \ref{Thm:Main}}\label{Section:Main}
We first prove a lemma. 
\begin{lemma}\label{Lem:New irreducible}
  Let $a,n,p\in \Z$ with $a\ge 1$, $n\ge 2$ and $p$ a prime. Then $f_{n,a,p}(x)$ is irreducible over $\Q$ except when $2\mid n$ and $p=a+1$. %If $2\nmid n$, then $f_{n,a,p}(x)$ is irreducible over $\Q$. If $2\mid n$, then $f_{n,a,p}(x)$ is irreducible over $\Q$ if and only if $a\ne p-1$. 
\end{lemma}
\begin{proof}
  If $2\mid n$ then $f_{n,p-1,p}(-1)=0$. We show that $f_{n,a,p}(x)$ is irreducible over $\Q$ in every other situation. Observe that $f_{n,p,p}(x)$ is Eisenstein with respect to $p$, and is therefore irreducible over $\Q$. If $p<a-1$, then $f_{n,a,p}(x)$ is irreducible over $\Q$ by the first part of \ref{Perron} in Lemma \ref{Lem:irreducible}. If $p=a-1$,  then $f_{n,a,p}(\pm 1)\ne 0$, and so $f_{n,a,p}(x)$ is irreducible over $\Q$ by the second part of \ref{Perron} in Lemma \ref{Lem:irreducible}. If $p>a-1$, then $f_{n,a,p}(x)$ is irreducible over $\Q$ by \ref{JonesActa} of Lemma \ref{Lem:irreducible}. 
  
  Therefore, we only need to address the case $p=a+1$ when $2\nmid n$. We claim in this case that all zeros $\alpha$ of $f_{n,p-1,p}(x)$ are such that $\abs{\alpha}>1$. Let 
  \begin{equation} \label{Eq:f(alpha)=0}
  f_{n,p-1,p}(\alpha)=\alpha^{n-1}(\alpha-(p-1))-p=0,
  \end{equation} and assume, by way of contradiction, that $\abs{\alpha}\le 1$. Suppose first that $\abs{\alpha}=1$, and let $\alpha=c+di\in \C$, so that $c^2+d^2=1$. Then, from \eqref{Eq:f(alpha)=0}, it follows that 
  \[p^2-\abs{c-(p-1)+di}^2=2(p-1)(c+1)=0.\] Hence, $c=-1$, which implies that $d=0$ and $\alpha=-1$, contradicting the fact that 
  \[f_{n,p-1,p}(-1)=(-1)^n-(p-1)(-1)^{n-1}-p=-2(p+1)\ne 0.\] Suppose then that $\abs{\alpha}<1$. Then, it follows from \eqref{Eq:f(alpha)=0} that 
  \[p=\abs{\alpha}^{n-1}\abs{\alpha-(p-1)}<\abs{\alpha-(p-1)}\le \abs{\alpha}+p-1<p,\] which is impossible. Thus, the claim that all zeros of $f_{n,p-1,p}(x)$ have modulus greater than 1 is established. 
  
  Assume now, by way of contradiction, that 
  $f_{n,p-1,p}(x)=g(x)h(x)$ is a nontrivial factorization of $f_{n,p-1,p}(x)$, where $g(x),h(x)\in \Z[x]$. Then, without loss of generality, we have $\abs{g(0)}=1$ and $\abs{h(0)}=p$, which is impossible since all zeros of $g(x)$ have modulus strictly larger than 1, and this final contradiction completes the proof of the lemma.      
\end{proof}
\begin{proof}[Proof of Theorem \ref{Thm:Main}]
    First, note that $f_{n,a,p}(x)$ is irreducible over $\Q$ by Lemma \ref{Lem:New irreducible}, except when $2\mid n$ and $p=a+1$. Second, since $\gcd(n,a(n-1))=1$, we have from Lemma \ref{Lem:Linear} with \[G(t):=n^nt+a^n(n-1)^{n-1},\] that there exist infinitely many primes $p$ such that $G(p)$ is squarefree. We will therefore assume that $p$ is such a prime, with $p\ne a+1$ if $2\mid n$, so that $f_{n,a,p}(x)$ is irreducible over $\Q$. 
  
   The remainder of the proof of the theorem is divided into the following two parts:
  \begin{enumerate}
    \item \label{I1} $f_{n,a,p}(x)$ is monogenic if and only if $G(p)$ is squarefree;
    \item \label{I2} if $p>a+1$ and $G(p)$ is squarefree, then $f_{n,a,p}(x)$ is a monogenic strictly-Perron polynomial.
  \end{enumerate} 
   To prove item \ref{I1}, we conduct a careful analysis of the conditions of Theorem \ref{Thm:JKS2} with $A=-a$, $B=-p$, $m=n-1$ and $q$ a prime divisor of $\Delta(f_{n,a,p})$. Recall from \eqref{Eq:Delta(f_p)} that 
   \begin{equation*}\label{Eq:Del}
   \Delta(f_{n,a,p})=(-1)^{(n-1)(n+2)/2}p^{n-2}G(p).
   \end{equation*} If $q\mid a$ and $q\mid p$, then $q=p$ and condition \ref{JKS:C1} 
   is clearly satisfied since $p$ is squarefree. If $q\mid a$ and $q\nmid b$, then $q\ne p$ so that $q\mid G(p)$, which is impossible, since $q\nmid n$. Hence, condition \ref{JKS:C2} is not applicable. Condition \ref{JKS:C3}  deals specifically with $q=p$, and since $a_2=0$ and $b_2=-1$, we see that the first statement under condition \ref{JKS:C3} is always satisfied. 
   Condition \ref{JKS:C4} deals with primes $q$ such that $q\nmid ap$ and $q\mid (n-1)$. But since $q\mid G(p)$, then $q\mid n^np$, which is impossible. Hence, condition \ref{JKS:C4} is not applicable. Finally, we are left with condition \ref{JKS:C5}, which handles primes $q$ dividing $G(p)$ such that $q\nmid ap(n-1)$. For such primes, we deduce from condition \ref{JKS:C5} that $q\mid \mid G(p)$ if and only if $f_p(x)$ is monogenic. We still need to examine the possible divisibility of $G(p)$ by the primes $q$ that do divide $ap(n-1)$. So, suppose that $q\mid G(p)$ and $q\mid ap(n-1)$. If $q\mid a$, then $q\nmid n$ since $\gcd(a,n)=1$, and therefore, $q=p$. Similar reasoning shows that if $q\mid(n-1)$, then $q=p$. That is, if $q\mid G(p)$ and $q\mid ap(n-1)$, then $q\mid a(n-1)$ and $q=p$. If $p\mid a(n-1)$, then $p^2\mid a^n(n-1)^n$ since $n\ge 2$. Consequently, $p\mid \mid G(p)$ in this case, and we have established item \ref{I1}.

   We turn now to item \ref{I2}. Since $G(p)$ is squarefree, we know from item \ref{I1} that $f_{n,a,p}(x)$ is monogenic. Observe that the companion matrix of $f_{n,a,p}(x)$ is 
   \[\CC=\begin{bmatrix}
    0 & 0 & \dots  & 0 & p \\
    1 & 0 & \dots & 0 & 0 \\
    0 & 1 & \dots &  0 & 0 \\
    \vdots & \vdots & \ddots & \vdots & \vdots\\
    0 & 0 & \dots & 1 & a
    \end{bmatrix}\] 
    and $\Gamma(\CC)$ is

 
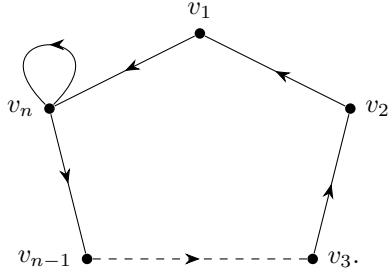
\[\begin{tikzpicture}[
       decoration = {markings,
                     mark=at position .5 with {\arrow{Stealth[length=2mm]}}},
       dot/.style = {circle, fill, inner sep=1.4pt, node contents={},
                     label=#1},
every edge/.style = {draw, postaction=decorate}
                        ]
\node (a) at (2,5) [dot=$v_1$];
\node (b) at (0,4) [dot=left:$v_n$];
\node (c) at (.5,2) [dot=left:$v_{n-1}$];
\node (d) at (3.5,2) [dot=right:$v_3$.];
\node (e) at (4,4) [dot=right:$v_2$];
\path   (a) edge (b)    (b) edge (c)    (d) edge (e)  (e) edge (a)
        (c) edge[dashed] (d)    (b) edge[loop, min distance=15mm] (b);
         \end{tikzpicture}\]
It is easy to see that $\Gamma(\CC)$ is strongly connected, and therefore, $\CC$ is irreducible by Theorem \ref{Thm:Irreducible}. Hence, since $f_{n,a,p}(x)$ is the characteristic polynomial of $\CC$, it follows from Lemma \ref{Lem:PFP} that $f_{n,a,p}(x)$ is a Perron polynomial. By Descartes' rule of signs, $f_{n,a,p}(x)$ has exactly one positive real zero $\lambda$, which must be the Perron number associated with $f_{n,a,p}(x)$. (Note that $f_{n,a,p}(1)<0$, which confirms that $\lambda>1$.) Hence, since $\lambda^{-1}$ is a positive real number, $\lambda^{-1}$ cannot be a zero of $f_{n,a,p}(x)$, and therefore, $\lambda$ is not a Salem number.  

By Descartes' rule of signs, we have that  
\[f_{n,a,p}(-x)=\left\{\begin{array}{cll}
  x^n+ax^{n-1}-p & \mbox{if $2\mid n$}& \Rightarrow \mbox{$f_{n,a,p}(x)$ has exactly one}\\
   & & \qquad \qquad \qquad \mbox{negative real zero;}\\[.5em]
  -x^n-ax^{n-1}-p & \mbox{if $2\nmid n$}& \Rightarrow \mbox{$f_{n,a,p}(x)$ has no}\\
   & & \qquad \qquad \qquad \mbox{negative real zeros.}\\
\end{array}\right.\] 
In the case that $2\mid n$, let $r$ be the unique negative real zero of $f_{n,a,p}(x)$. Observe that  $f_{n,a,p}(-1)=1+a-p<0$ since $p>a+1$, and thus, $\abs{r}>1$. Consequently, regardless of the parity of $n$, if $\alpha$ is a zero of $f_{n,a,p}(x)$ with $\abs{\alpha}<1$, then $n\ge 3$ and $\alpha\not \in \R$, so that $f_{n,a,p}(x)$ has at least two distinct zeros with modulus less than 1. Hence, $\lambda$ cannot be an anti-Pisot number. 

To show that $\lambda$ is not a Pisot number, we show that $f_{n,a,p}(x)$ has a zero $\alpha\ne \lambda$ with $\abs{\alpha}>1$. If $2\mid n$, then we have already shown that $f_{n,a,p}(x)$ has a real zero $r\ne \lambda$ with $\abs{r}>1$. Hence, suppose that $2\nmid n$. In this case, we know that $f_{n,a,p}(x)$ has exactly one real zero $\lambda$. Let $z_1,z_2,\ldots ,z_{n-1}$ be the remaining zeros of $f_{n,a,p}(x)$ where $z_i\ne \lambda$ and $z_i\not \in \R$. Suppose, by way of contradiction, that $\abs{z_i}<1$ for all $i$. 
Then,  
\begin{equation}\label{Eq:contra}
\lambda>\lambda \prod_{i=1}^{n-1}\abs{z_i}=\abs{f_{n,a,p}(0)}=p.
\end{equation} Hence, since $p>a+1$ and $n\ge 2$, we have that   
  \[f_{n,a,p}(x)=x^{n-1}(x-a)-p\ge p^{n-1}(p-a)-p>0\] for all $x\ge p$. Thus, it follows from \eqref{Eq:contra} that $f_{n,a,p}(\lambda)>0$, which contradicts the fact that $\lambda$ is a zero of $f_{n,a,p}(x)$. Hence, there exists some zero $\alpha\ne \lambda$ of $f_{n,a,p}(x)$ with $\abs{\alpha}>1$, and we conclude that $\lambda$ is not a Pisot number, which completes the proof of the theorem. 
\end{proof}

%\section*{Acknowledgments}

%\section*{Data Availability Statement}
%The author confirms that all relevant data are included in the article.

\end{document}